\documentclass[pdflatex,sn-mathphys-num]{sn-jnl}
\usepackage{graphicx}%
\usepackage{multirow}%
\usepackage{amsmath,amssymb,amsfonts}%
\usepackage{amsthm}%
\usepackage{mathrsfs}%
\usepackage[title]{appendix}%
\usepackage{xcolor}%
\usepackage{textcomp}%
\usepackage{manyfoot}%
\usepackage{booktabs}%
\usepackage{algorithm}%
\usepackage{algorithmicx}%
\usepackage{algpseudocode}%
\usepackage{listings}%

\newcommand{\R}{\mathbb{R}}

\theoremstyle{thmstyleone}%
\newtheorem{theorem}{Theorem}%
\newtheorem{proposition}[theorem]{Proposition}%
\newtheorem{lemma}[theorem]{Lemma}
\newtheorem{corollary}[theorem]{Corollary}

\theoremstyle{thmstyletwo}%
\newtheorem{remark}{Remark}%

\theoremstyle{thmstylethree}%

\begin{document}

\title[$L^2$-supercritical Kirchhoff equation with partial confinement]{Multiplicity of normalized solutions for the $L^2$-supercritical Kirchhoff equation with partial confinement}


\author[1]{\fnm{Qun} \sur{Wang}}\email{wangq199@nenu.edu.cn}

\author*[2]{\fnm{Xiaojun} \sur{Chang}}\email{changxj100@nenu.edu.cn}

\affil[1]{\orgdiv{School of Mathematics and Statistics}, \orgname{Northeast Normal University}, \orgaddress{\city{Changchun}, \postcode{130024}, \state{Jilin}, \country{PR China}}}

\affil*[2]{\orgdiv{School of Mathematics and Statistics \& Center for Mathematics and Interdisciplinary Sciences}, \orgname{Northeast Normal University}, \orgaddress{\city{Changchun}, \postcode{130024}, \state{Jilin}, \country{PR China}}}

\abstract{We study the existence of positive normalized solutions for the nonlocal Kirchhoff equation with partial harmonic confinement
\[-\left(a+b\int_{\mathbb{R}^{3}}|\nabla u|^{2}\,dx\right)\Delta u
+(x_1^2+x_2^2)u+\lambda u=|u|^{p-2}u,\]
subject to the mass constraint $\int_{\mathbb{R}^{3}}|u|^2\,dx=\mu$,
where $a,b>0$ are constants, $\frac{14}{3}<p<6$, $\mu>0$ is a prescribed mass, and $\lambda\in\mathbb{R}$ is a Lagrange multiplier associated with the mass constraint. For every sufficiently small $\mu>0$, we prove the existence of two positive normalized solutions. One of these solutions is a local minimizer, while the other is obtained at a  mountain-pass level. The construction of the local minimizer relies on a combination of profile decomposition and a new nonlocal merging technique, yielding compactness up to translations in the unconfined $x_3$-direction. The mountain-pass solution is obtained by considering positive solutions defined on expanding balls. A uniform bound on the Morse index allows us to obtain a spectral estimate for the limiting Lagrange multiplier. Together with an energy comparison, this spectral estimate yields strong convergence in the energy space and recovers both the prescribed mass
and the Kirchhoff coefficient.}




\maketitle

\section{Introduction and main result}\label{sec1}

In this paper, we study the Kirchhoff equation with partial
harmonic confinement
\begin{equation}\label{1}
-\left(a+b\int_{\R^3}|\nabla u|^2dx\right)\Delta u+(x_1^2+x_2^2)u+\lambda u=|u|^{p-2}u
\end{equation}
subject to the mass constraint
\begin{equation}\label{constraint}
\int_{\R^3}|u|^2dx=\mu,
\end{equation}
where $a,b>0$, $\frac{14}{3}<p<6$, and $\mu>0$ is the prescribed mass.
The parameter $\lambda\in\mathbb R$ arises as a Lagrange
multiplier associated with the mass constraint.

The nonlocal principal part in \eqref{1} is motivated by
Kirchhoff's model \cite{Kirchhoff-1883} for the transverse
vibrations of an elastic string:
\[
\rho\frac{\partial^2u}{\partial t^2}
-\bigg(\frac{P_0}{h}+\frac{E}{2L}\int_{0}^{L}\bigg|\frac{\partial u}{\partial x}\bigg|^2dx\bigg)
\frac{\partial^2u}{\partial x^2}=0.
\]
Here, $\rho, P_0, h, E$ and $L$ denote the mass density, initial tension,
cross-sectional area, Young's modulus, and 
the length of the string, respectively.
The mathematical analysis of Kirchhoff-type equations was further developed by Pohozaev \cite{Pohozaev-1975} and Lions \cite{Lions-1978}. Results on well-posedness and global solvability for the associated
evolution equations were obtained in
\cite{Arosio-1996,D'Ancona-1992}.

For stationary Kirchhoff equations, existence results on bounded
domains were obtained in \cite{LLS-2014,Perera-2006}. Problems in the whole space with potentials were 
studied in
\cite{Deng-2015,Figueiredo-2014,He-2012,Li-2020}, with results on
existence, multiplicity, concentration and asymptotic behavior.
In \eqref{1}, the coefficient of $-\Delta u$ depends on the total
kinetic energy of the solution. This dependence gives rise to
the quartic kinetic term in the corresponding energy functional.

We seek normalized solutions of \eqref{1}--\eqref{constraint}
as critical points of the energy functional
\[
J(u):=
\frac{a}{2}\int_{\mathbb R^3}|\nabla u|^2\,dx
+\frac{b}{4}\left(\int_{\mathbb R^3}|\nabla u|^2\,dx\right)^2
+\frac{1}{2}\int_{\mathbb R^3}(x_1^2+x_2^2)|u|^2\,dx
-\frac{1}{p}\int_{\mathbb R^3}|u|^p\,dx
\]
restricted to the mass sphere
\[
S_\mu:=\left\{u\in H:
\int_{\mathbb R^3}|u|^2\,dx=\mu\right\},
\]
where
\[
H:=\left\{u\in H^1(\mathbb R^3,\mathbb R):
\int_{\mathbb R^3}(x_1^2+x_2^2)|u|^2\,dx<\infty\right\}
\]
is equipped with the inner product
\[
(u,v)_H:=\int_{\mathbb R^3}
\left(\nabla u\cdot\nabla v+(x_1^2+x_2^2)uv+uv\right)\,dx
\]
and the induced norm
\[
\|u\|_H^2:=\int_{\mathbb R^3}
\left(|\nabla u|^2+(x_1^2+x_2^2)|u|^2+|u|^2\right)\,dx.
\]

The range $\frac{14}{3}<p<6$ is $L^2$-supercritical for the
Kirchhoff energy and Sobolev subcritical in $\mathbb{R}^3$. For $t>0$ and $u\in H$, the dilation
$(t\star u)(x):=t^{\frac32}u(tx)$ preserves the $L^2$ norm.
Under this dilation, the kinetic terms scale as $t^2$ and $t^4$,
the potential term as $t^{-2}$, and the nonlinear term as
$t^{\frac{3(p-2)}{2}}$.
Since $\frac{3(p-2)}2>4$, we have $J(t\star u)\to-\infty$ as
$t\to\infty$.
Thus, $J$ is unbounded from below on $S_\mu$ for every $\mu>0$,
and there is no global constrained minimizer in this range.

For Kirchhoff equations in the absence of an external potential,
Ye \cite{Ye-2015} studied the sharp existence of global constrained minimizers
in the $L^2$-subcritical regime and obtained mountain-pass solutions
in the $L^2$-supercritical regime for $1\leq N\leq3$.
Qi and Zou \cite{Qi-2022}  subsequently classified the positive normalized
solutions for pure-power nonlinearities. Further existence results were obtained for general
$L^2$-supercritical nonlinearities in \cite{He-2023}
and for combined nonlinearities involving a Sobolev critical
term in \cite{Chen-2024}.

In the presence of trapping potentials satisfying $V(x)\to+\infty$ as
$|x|\to\infty$, the existence and asymptotic behavior of
normalized solutions for Kirchhoff equations were studied in
\cite{Guo-2018,Li-2019,Hu-2021}.
For suitable potentials satisfying $V(x)\to0$ as $|x|\to\infty$,
Cai and Zhang \cite{Cai-2023} obtained normalized solutions
in $\mathbb R^3$ for $\frac{14}{3}<p<6$, while
He, Lv and Tang \cite{He-JGA-2023} treated combined
$L^2$-supercritical and Sobolev critical nonlinearities.

The present work is motivated by the corresponding semilinear
problem with partial harmonic confinement.
 Bellazzini, Boussa\"id, Jeanjean and
Visciglia \cite{Bellazzini-2017} studied the nonlinear Schr\"odinger equation
\begin{equation*}\label{jeanjean}
-\Delta u+(x_1^2+x_2^2)u+\lambda u=|u|^{p-2}u
\quad\text{in }\mathbb R^3,
\end{equation*}
subject to the mass constraint \eqref{constraint} in the $L^2$-supercritical range
$\frac{10}{3}<p<6$. Using the concentration-compactness principle of Lions \cite{Lions-1984-1,Lions-1984-2}, they established the existence of a positive local
minimizer for sufficiently small mass and studied its stability
and qualitative properties. Their work inspired a number of
studies on various partially confined nonlinear models, including the Hartree--Fock system \cite{Jia-2021},
nonlinear Schr\"odinger systems \cite{Jia-2020,Wang-2019}, and
inhomogeneous Gross--Pitaevskii equations \cite{LiuHeFeng}.

Bellazzini et al.\ also raised the question of the existence of a second
solution at a mountain-pass level. Wei and Wu \cite{Wei-2023}
gave a positive answer for sufficiently small mass through a reduction
argument. Jeanjean and Song \cite{Jeanjean-2025} subsequently developed
a non-perturbative variational approach based on solutions
on expanding balls; see also \cite{BQZ-2024} for related
constructions.
Their proof combines uniform bounds derived from the Pohozaev
identity, compactness arguments based on symmetry and
monotonicity, and spectral information on the limiting
Lagrange multiplier.

Motivated by these results, we investigate whether a local
minimizer and a mountain-pass solution coexist for the
normalized Kirchhoff problem \eqref{1}-\eqref{constraint}.  As in the semilinear case, the loss of compactness still originates from translations in the unconfined 
 $x_3$-direction. However, the diffusion coefficient now depends on the total kinetic energy, which introduces additional difficulties
in both variational constructions.

For convenience, write $x_\perp=(x_1,x_2)$ and set
\[
A(u):=\int_{\mathbb R^3}|\nabla u|^2\,dx,\qquad
C(u):=\int_{\mathbb R^3}|x_\perp|^2|u|^2\,dx,\qquad
P(u):=\int_{\mathbb R^3}|u|^p\,dx.
\]
Then
\[
J(u)=\frac a2A(u)+\frac b4A(u)^2
+\frac12C(u)-\frac1pP(u).
\]

For the local minimizer, the compactness argument in the
semilinear case relies on a strict comparison of constrained
minimum levels at different masses.
To explain the difficulty in the Kirchhoff case, let
$u\in S_\mu$ and $\theta>1$.
Then $\sqrt{\theta}\,u\in S_{\theta\mu}$ and
\begin{equation*}\label{eq:intro-amplitude}
J(\sqrt{\theta}\,u)-\theta J(u)
=
\frac b4(\theta^2-\theta)A(u)^2
-\frac{\theta^{\frac{p}{2}}-\theta}{p}P(u).
\end{equation*}
When $b=0$, the right-hand side is negative, which is a
key ingredient in the comparison of the semilinear
minimum levels; see
\cite[Lemma~3.5]{Bellazzini-2017}.
For $b>0$, the additional positive term means that
the same scaling argument does not directly yield
the required comparison.
We overcome this difficulty by combining profile decomposition
with a nonlocal merging argument that rules out splitting into
multiple nonzero profiles.

For the mountain-pass solution, we adapt the approximation
on expanding balls developed in \cite{Jeanjean-2025}.
The main issue is to pass to the limit while recovering both
the prescribed mass and the Kirchhoff coefficient. The spectral threshold entering the compactness argument
now depends on the a priori unknown limit of the kinetic
energy.
We obtain the required spectral estimate from the Morse
index bound supplied by the constrained minimax construction
and combine it with an energy comparison to establish
strong convergence.

Our main result is the following.
\begin{theorem}\label{thm:main}
Let $a,b>0$ and $\frac{14}{3}<p<6$.
Then there exists $\mu_*>0$ such that, for every $\mu\in(0,\mu_*)$, the constrained functional $J|_{S_\mu}$ admits at least two positive critical points $u_1, u_2\in S_\mu$, where $u_1$ is a local minimizer and $u_2$ is a critical point at a mountain-pass level.  Moreover,
\[
                0<J(u_1)<d_\mu<J(u_2),
\]
where
\begin{eqnarray}\label{eq:vanishing-threshold}
 d_\mu:=\min_{s>0}
      \left(\frac a2s+\frac b4s^2+\frac{\mu^2}{2s}\right).
\end{eqnarray}
\end{theorem}

The level $d_\mu$ provides a common reference for the two
variational constructions. We prove that
\[
d_\mu=\inf_{u\in S_\mu}J_0(u),
\qquad
J_0(u):=\frac a2A(u)+\frac b4A(u)^2+\frac12C(u),
\]
and that this infimum is not attained. In particular,
every sequence $\{v_n\}\subset S_\mu$ satisfying
$P(v_n)\to0$ has
$\liminf\limits_{n\to\infty}J(v_n)\ge d_\mu$.
When $b=0$, this threshold reduces to
$d_\mu=\sqrt a\,\mu$.

We construct the local minimizer without imposing symmetry. For a fixed $\chi>0$, set
\begin{align*}
&\mathfrak q_a(u):=aA(u)+C(u),\\
&c_\mu:=\inf\{J(u):u\in S_\mu,\ \mathfrak q_a(u)<\chi\}.
\end{align*}
For small mass, we show that $c_\mu<d_\mu$, which excludes
vanishing. We apply a profile decomposition \cite{HmidiKeraani}
 to a suitable minimizing
Palais--Smale sequence and fix the Kirchhoff coefficient
at the limit of the total kinetic energy, as in \cite{Li-2014,QiuOuLv2024}. The nonzero profiles separating in the $x_3$-direction are positive solutions of a common
limiting equation that retains the transverse harmonic
potential.
Their positivity yields a strict spectral gap, which
allows us to eliminate the remainder and recover the
prescribed mass and the total kinetic and potential
energies. Denoting these profiles by $\{U^j\}_{j\ge1}$, we consider
the merged function
\[
W:=\left(\sum_{j\ge1}(U^j)^2\right)^{1/2}.
\]
This function preserves the total mass and potential
energy without increasing the kinetic energy.
If more than one profile is nonzero, its $L^p$
term is strictly larger than the sum of the nonlinear
terms of the profiles.
Thus $W$ belongs to the same local admissible set and
would have energy strictly below $c_\mu$, a contradiction.
This proves compactness up to translations in the
$x_3$-direction.

For the second solution, we invoke the constrained minimax theorem from \cite{Borthwick-2024}, which yields positive mountain-pass solutions on the balls
\(B_R:=B_R(0)=\{x\in\mathbb{R}^3:|x|<R\}\), \(R>0\), together with a uniform bound on the Morse index.
 Related applications to normalized problems can be found
in \cite{BCJS-2023,CJS-2024,CRZ-2025,Wang-2024}. In our approximation, the Pohozaev identity yields
uniform bounds in $H$,
while symmetry and monotonicity obtained by moving plane
arguments \cite{Gidas-1979} give strong convergence 
in $L^p(\mathbb R^3)$.
 Along a subsequence with
$A(u_R)\to K_\infty$ and $\lambda_R\to\lambda_\infty$, the Morse index bound yields
\begin{equation}\label{eq:intro-spectral}
\lambda_\infty\geq-2\sqrt{a+bK_\infty}.
\end{equation}
Together with the equation and the uncertainty inequality
in the confined variables, this estimate shows that a
zero weak limit would force the energies to converge
to $d_\mu$.
Since the limiting mountain-pass level satisfies
$m_\mu>d_\mu$, the weak limit is nonzero.
Its positivity then yields a strict spectral gap and
strong convergence in $H$, recovering both the prescribed
mass and the Kirchhoff coefficient.

The paper is organized as follows. Section~\ref{Preliminaries} collects the spectral and compactness tools. 
Sections~\ref{Sec:Local} and~\ref{Sec:Mountain} are devoted to the local minimizer
and the mountain-pass solution, respectively. Throughout the paper, we write $u^-:=\max\{-u,0\}$ for a real-valued function $u$, and denote by $\|\cdot\|_q$ the norm in $L^q(\mathbb R^3)$ for $2\le q\le 6$.

\section{Preliminaries}\label{Preliminaries}

We begin with the Gagliardo--Nirenberg inequality:
for every $p\in(2,6)$, there exists $C_p>0$ such that
\begin{equation}\label{eq:GN}
P(u)\leq C_p\|u\|_2^{\frac{6-p}{2}}A(u)^{\frac{3p-6}{4}},
\qquad u\in H^1(\mathbb R^3).
\end{equation}

We write
\[
x=(x_\perp,x_3)\in\mathbb R^2\times\mathbb R,\quad
x_\perp:=(x_1,x_2),\quad
\nabla_\perp=(\partial_{x_1},\partial_{x_2}).
\]
For $z\in\mathbb R$, define the translation in the unconfined direction by
\[
(\tau_z u)(x_\perp,x_3):=u(x_\perp,x_3-z).
\]
The $H$-norm and the functional $J$ are invariant under $\tau_z$.

For $u\in C_c^\infty(\mathbb R^3)$, using integration by parts in the confined variables and the Cauchy--Schwarz inequality,
we have
\[
\|u\|_2^2
=
-\int_{\mathbb R^3}
u\,x_\perp\cdot\nabla_\perp u\,dx
\le
\|\nabla_\perp u\|_2\,C(u)^{\frac{1}{2}}.
\]
By density, we obtain the uncertainty inequality in the confined variables
\begin{equation}\label{eq:partial-uncertainty}
\left(\int_{\mathbb R^3}|\nabla_\perp u|^2\,dx\right)C(u)
\ge \|u\|_2^4,
\quad u\in H.
\end{equation}
Since $\|\nabla_\perp u\|_2^2\le A(u)$, we have
$A(u)C(u)\ge\|u\|_2^4$. Consequently, for every $M>0$
and $u\in H$,
\begin{equation}\label{eq:spectral-ineq}
MA(u)+C(u)
\ge 2\sqrt{M}\sqrt{A(u)C(u)}
\ge 2\sqrt{M}\|u\|_2^2.
\end{equation}\\

We next determine the bottom of the spectrum of the
partially confined operator.
 Define the energy space by
\[
H_\perp:=
\left\{
v\in H^1(\mathbb R^2;\mathbb R):
\int_{\mathbb R^2}|x_\perp|^2|v|^2\,dx_\perp<\infty
\right\}.
\]
For $M>0$, set
\[
\begin{aligned}
\Lambda(M)
&:=
\inf_{\substack{u\in H\\
\|u\|_{L^2(\mathbb R^3)}=1}}
\bigl(MA(u)+C(u)\bigr),\\
\lambda_\perp(M)
&:=
\inf_{\substack{v\in H_\perp\\
\|v\|_{L^2(\mathbb R^2)}=1}}
\left(
M\int_{\mathbb R^2}|\nabla_\perp v|^2\,dx_\perp
+\int_{\mathbb R^2}|x_\perp|^2|v|^2\,dx_\perp
\right).
\end{aligned}
\]

\begin{lemma}\label{lem:spectral-bottom}
For every $M>0$,
\begin{equation*}
\Lambda(M)=\lambda_\perp(M)=2\sqrt{M}.
\end{equation*}
Moreover, $\lambda_\perp(M)$ is attained by
\begin{equation*}
\Phi_M(x_\perp):=(\pi\sqrt{M})^{-\frac{1}{2}}
\exp\left(-\frac{|x_\perp|^2}{2\sqrt{M}}\right).
\end{equation*}
\end{lemma}
\begin{proof}
For \(M=1\), applying Lemma 2.1 in \cite{Bellazzini-2017}, we obtain
\[
\Lambda(1)=\lambda_\perp(1)=2
\]
(see also \cite[Lemma 2.4 and Remark 2.5]{Jeanjean-2025}). Moreover, the value \(\lambda_\perp(1)\) is attained by the normalized Gaussian
\[
\Phi_1(x_\perp)=\pi^{-\frac12}\exp\left(-\frac{|x_\perp|^2}{2}\right).
\]

For \(d\in\{2,3\}\), define
\[
(T_M^{(d)}f)(x):=M^{-\frac{d}{8}}f\left(M^{-\frac{1}{4}}x\right).
\]
In the formulas below, $x_\perp=(x_1, x_2)$ when $d=3$, whereas $x_\perp=x$ when $d=2$. The map $T_M^{(d)}$ preserves the $L^2$ norm, and a change
of variables gives
\[
\|T_M^{(d)}f\|_{L^2(\mathbb R^d)}=\|f\|_{L^2(\mathbb R^d)}
\]
and
\[
M\int_{\mathbb R^d}\left|\nabla T_M^{(d)}f\right|^2\,dx
+\int_{\mathbb R^d}|x_\perp|^2\left|T_M^{(d)}f\right|^2\,dx
=\sqrt M\left(\int_{\mathbb R^d}|\nabla f|^2\,dx+\int_{\mathbb R^d}|x_\perp|^2|f|^2\,dx\right).
\]
Since \(T_M^{(d)}\) is a bijection on the relevant energy space, we have
\[
\Lambda(M)=\sqrt M\,\Lambda(1),\qquad
\lambda_\perp(M)=\sqrt M\,\lambda_\perp(1).
\]
Combining this with \(\Lambda(1)=\lambda_\perp(1)=2\), we obtain
\[
\Lambda(M)=\lambda_\perp(M)=2\sqrt M.
\]

Finally, a direct computation yields
\[T_M^{(2)}\Phi_1(x_\perp)
=(\pi\sqrt M)^{-\frac{1}{2}}\exp\left(-\frac{|x_\perp|^2}{2\sqrt M}\right)=\Phi_M(x_\perp).
\]
Thus, $\lambda_\perp(M)$ is attained at $\Phi_M$.
\end{proof}

Using the transverse ground state, we establish the following strict lower bound for the Lagrange multiplier, which plays a crucial role in the compactness arguments below.

\begin{lemma}\label{lem:strict-edge}
Let $M>0$ and $r\in(2,6]$. If $0\neq u\in H$ is a nonnegative weak solution of
\begin{equation}\label{M EQ}
-M\Delta u+|x_\perp|^2u+\lambda u=|u|^{r-2}u\quad\text{in }\mathbb{R}^{3},
\end{equation}
then $\lambda>-2\sqrt{M}$.
\end{lemma}

\begin{proof}
Let $\Phi_M>0$ be given in Lemma~\ref{lem:spectral-bottom}, so that
\begin{equation*}
-M\Delta_\perp\Phi_M+|x_\perp|^2\Phi_M=2\sqrt M\,\Phi_M\quad\text{in }\mathbb R^2,
\quad\|\Phi_M\|_{L^2(\mathbb R^2)}=1.
\end{equation*}
For almost every $x_3\in\R$, set
\[
h(x_3):=\int_{\mathbb R^2}u(x_\perp,x_3)\Phi_M(x_\perp)\,dx_\perp
\]
and
\[
g(x_3):=\int_{\mathbb R^2}u(x_\perp,x_3)^{r-1}\Phi_M(x_\perp)\,dx_\perp.
\]

By Cauchy--Schwarz inequality it follows that
\[
\|h\|_{L^2(\R)}\leq\|u\|_2,\qquad
\|h'\|_{L^2(\R)}\leq\|\partial_3u\|_2,
\]
where
\[
h'(x_3)=\int_{\R^2}\partial_3u(x_\perp,x_3)\Phi_M(x_\perp)\,\mathrm dx_\perp
\]
in the distributional sense. Hence $h\in H^1(\mathbb R)$.

Since $u\in L^r(\mathbb R^3)$ and $\Phi_M\in L^r(\mathbb R^2)$, using H\"older's inequality, we get
\[
|g(x_3)|\leq\left(\int_{\mathbb R^2}|u(x_\perp,x_3)|^r\,dx_\perp\right)^{\frac{r-1}{r}}
\|\Phi_M\|_{L^r(\mathbb R^2)}.
\]
Thus, it follows that
\[
\|g\|_{L^{\frac{r}{r-1}}(\mathbb R)}\leq\|\Phi_M\|_{L^r(\mathbb R^2)}\|u\|_r^{r-1}<+\infty.
\]
Consequently, $g\in L^{\frac r{r-1}}(\mathbb R)$ and, in particular,
$g\in L^1_{\mathrm{loc}}(\mathbb R)$. Since $u\geq0$, $u\not\equiv0$, and $\Phi_M>0$, we also have
$h\geq0$, $g\geq0$ and $g\not\equiv0$.

Let $\varphi\in C_c^\infty(\mathbb R)$. Testing \eqref{M EQ} with $\Phi_M(x_\perp)\varphi(x_3)$, we obtain
\[
M\int_{\mathbb R}h'\varphi'\,dx_3+(2\sqrt M+\lambda)\int_{\mathbb R}h\varphi\,dx_3
=\int_{\mathbb R}g\varphi\,dx_3.
\]
Therefore, $h$ satisfies
\begin{equation}\label{eq:projected-equation}
-Mh''+(2\sqrt M+\lambda)h=g
\quad\text{in }\mathcal D'(\mathbb R).
\end{equation}
Here, $\mathcal D'(\mathbb R)$ denotes the space of distributions on $\mathbb R$.

Suppose, by contradiction, that $\lambda\leq-2\sqrt M$.
Since $h\in H^1(\mathbb R)$ and $g\in L^1_{\mathrm{loc}}(\mathbb R)$,
\eqref{eq:projected-equation} implies that
$h\in W^{2,1}_{\mathrm{loc}}(\mathbb R)$.
Moreover, using $h\geq0$ and $g\geq0$, we obtain
\[
-Mh''=g-(2\sqrt M+\lambda)h\geq0
\qquad\text{a.e. in }\mathbb R.
\]
Thus $h'$ is locally absolutely continuous and nonincreasing,
so $h$ is concave.
Since $h$ is nonnegative on $\mathbb R$, it must be constant.
The fact that $h\in L^2(\mathbb R)$ then gives $h\equiv0$.
Returning to \eqref{eq:projected-equation}, we obtain $g\equiv0$,
contrary to $g\not\equiv0$.
Therefore, $\lambda>-2\sqrt M$.
\end{proof}

We next give a variational characterization of $d_\mu$, which
will be used to rule out vanishing of minimizing sequences.
For $u\in S_\mu$, set
\[
J_0(u):=\frac{a}{2}A(u)+\frac{b}{4}A(u)^2+\frac{1}{2}C(u).
\]
Define
\[
f_\mu(s):=\frac{a}{2}s+\frac{b}{4}s^2+\frac{\mu^2}{2s},
\qquad s>0,
\]
so that $d_\mu=\min_{s>0}f_\mu(s)$.
The equation
\begin{equation}\label{eq:s-mu-equation}
as^2+bs^3=\mu^2
\end{equation}
has a unique positive solution, denoted by $s_\mu$.

To approximate $d_\mu$, we introduce a family of test functions
that spread in the unconfined $x_3$-direction.
Let $\Phi_1$ be given by Lemma~\ref{lem:spectral-bottom}. By direct computations, we get
\begin{equation*}
\int_{\mathbb{R}^{2}}|\Phi_1|^2\,dx_\perp
=\int_{\mathbb{R}^{2}}|\nabla_\perp\Phi_1|^2\,dx_\perp
=\int_{\mathbb{R}^{2}}|x_\perp|^2|\Phi_1|^2\,dx_\perp=1.
\end{equation*}
For $\sigma>0$, define
\[
\Psi_\sigma(x_\perp):=\sigma\Phi_1(\sigma x_\perp),\quad\sigma_\mu^2:=\frac{s_\mu}{\mu}.
\]
Choose $0\leq\eta\in C_c^\infty(\mathbb{R})$ with $\|\eta\|_2=1$ and set
\begin{equation}\label{eq:test-family}
\eta_t(x_3):=t^{\frac{1}{2}}\eta(tx_3),\quad
\psi_{\mu,t}(x_\perp,x_3):=\sqrt{\mu}\,\Psi_{\sigma_\mu}(x_\perp)\eta_t(x_3).
\end{equation}
Obviously, $\psi_{\mu,t}\in S_\mu$. Moreover, we have
\begin{align}
A(\psi_{\mu,t})&=s_\mu+\mu t^2\int_{\mathbb{R}}|\eta'|^2\,dx_3,\label{eq:test-A}\\
C(\psi_{\mu,t})&=\frac{\mu^2}{s_\mu}=a s_\mu+b s_\mu^2,\label{eq:test-C}\\
P(\psi_{\mu,t})&=\mu s_\mu^{\frac{p-2}{2}}t^{\frac{p-2}{2}}
\left(\int_{\mathbb{R}^{2}}|\Phi_1|^p\,dx_\perp\right)
\left(\int_{\mathbb{R}}|\eta|^p\,dx_3\right).
\label{eq:test-P}
\end{align}
Since $p>2$, \eqref{eq:test-P} yields $P(\psi_{\mu,t})\to0$ as $t\to0^+$.

\begin{lemma}\label{lem:vanishing-level}$s_\mu$ is the unique minimizer of $f_\mu$, and $d_\mu$, defined in
\eqref{eq:vanishing-threshold}, satisfies
\begin{equation}\label{eq:d-mu-formula}
d_\mu=a s_\mu+\frac{3b}{4}s_\mu^2.
\end{equation}
Moreover, $d_\mu=\inf\limits_{u\in S_\mu}J_0(u)$ is not attained, and $J_0(\psi_{\mu,t})\to d_\mu$ as $t\to0^+$.
As $\mu\to0^+$, we have
\begin{equation}\label{eq:s-d-asymptotics}
\frac{s_\mu}{\mu}\to\frac{1}{\sqrt{a}},\quad\frac{d_\mu}{\mu}\to\sqrt{a}.
\end{equation}
\end{lemma}

\begin{proof}
By direct computations, we get
\begin{equation*}
f_\mu'(s)=\frac{a}{2}+\frac{b}{2}s-\frac{\mu^2}{2s^2},
\quad f_\mu''(s)=\frac{b}{2}+\frac{\mu^2}{s^3}>0.
\end{equation*}
Moreover, it follows that
\[
f_\mu'(s)\to-\infty\quad\text{as }s\to0^+,\quad f_\mu'(s)\to+\infty\quad\text{as }s\to+\infty.
\]
Since $f_\mu''(s)>0$, the derivative $f_\mu'$ is strictly increasing and hence has a unique zero. Consequently, $f_\mu$ possesses a unique critical point $s_\mu$, which is its unique global minimizer. Using \eqref{eq:s-mu-equation} at $s=s_\mu$, we obtain \eqref{eq:d-mu-formula}.

For every $u\in S_\mu$, we have $A(u)>0$. Thus, by \eqref{eq:partial-uncertainty},
\begin{equation*}
C(u)\geq\frac{\mu^2}{A(u)},
\end{equation*}
and hence
\begin{equation*}
J_0(u)\geq f_\mu(A(u))\geq d_\mu.
\end{equation*}
On the other hand, using \eqref{eq:test-A} and \eqref{eq:test-C}, we get
\begin{equation*}
J_0(\psi_{\mu,t})\to\frac{a}{2}s_\mu+\frac{b}{4}s_\mu^2+\frac{\mu^2}{2s_\mu}=d_\mu
\end{equation*}
as $t\to0^+$.

Suppose that $d_\mu$ is attained by some $u\in S_\mu$. Then, we have $A(u)C(u)=\mu^2$. Since $C(u)>0$, it follows from \eqref{eq:partial-uncertainty} that
\begin{equation*}
\mu^2\leq\|\nabla_\perp u\|_2^2C(u)\leq A(u)C(u)=\mu^2.
\end{equation*}
Thus, $\|\partial_{x_3}u\|_2=0$, implying that $u$ is independent of $x_3$ almost everywhere. As $u\in L^2(\mathbb{R}^{3})$, it follows that $u=0$, which contradicts $\|u\|_2^2=\mu>0$. Consequently, $d_\mu$ is not attained.

Finally, dividing \eqref{eq:s-mu-equation} by $\mu^2$ and setting
$t_\mu:=\frac{s_\mu}{\mu}$, we obtain $a t_\mu^2+b\mu t_\mu^3=1$,
and hence $0<t_\mu\leq a^{-\frac{1}{2}}$. Therefore, we have
\[
0\leq1-a t_\mu^2=b\mu t_\mu^3\leq b\mu a^{-\frac{3}{2}}\to0,
\]
which means
\[
\frac{s_\mu}{\mu}=t_\mu\to\frac{1}{\sqrt a}.
\]
Then, using \eqref{eq:d-mu-formula}, we obtain
\[
\frac{d_\mu}{\mu}=a\frac{s_\mu}{\mu}+\frac{3b}{4}\mu\left(\frac{s_\mu}{\mu}\right)^2\to\sqrt a.
\]
\end{proof}

For the local minimizer, where no symmetry is imposed, we use a profile decomposition in which translations  are allowed only along the unconfined $x_3$-direction. This is a standard consequence of the profile decomposition in $H^1(\R^3)$. See \cite[Proposition 3.1]{HmidiKeraani}. For the partially confined setting, we refer the reader in particular to \cite[Lemma 2.6]{LiuHeFeng}.

\begin{lemma}\label{lem:profile}
Let $\{v_n\}$ be bounded in $H$. After passing to a subsequence, still denoted by $\{v_n\}$, there exist profiles $\{U^j\}_{j\geq1}\subset H$, sequences $\{z_n^j\}_{n\geq1}\subset\mathbb R$, indexed by $j\geq1$, and remainders $r_n^\ell\in H$ such that the following properties hold.

\begin{enumerate}
\item[(i)]
For every $j\neq k$, $|z_n^j-z_n^k|\to+\infty$ as $n\to\infty$.

\item[(ii)]
For every fixed $\ell\geq1$,
\[
v_n=\sum_{j=1}^{\ell}\tau_{z_n^j}U^j+r_n^\ell\quad\text{in }H.
\]
Moreover, for every $q\in(2,6)$,
\begin{equation}\label{eq:remainder-Lq}
\lim_{\ell\to\infty}\limsup_{n\to\infty}\|r_n^\ell\|_q=0.
\end{equation}

\item[(iii)]
For every fixed $j\geq1$,
\begin{equation}\label{eq:shifted-weak}
\tau_{-z_n^j}v_n\rightharpoonup U^j\quad\text{in }H.
\end{equation}
Moreover, for every fixed $\ell\geq1$ and every $j\leq\ell$,
\[
\tau_{-z_n^j}r_n^\ell\rightharpoonup0\quad\text{in }H.
\]

\item[(iv)]
For every fixed $\ell\geq1$, as $n\to\infty$,
\begin{align*}
\|v_n\|_2^2&=\sum_{j=1}^{\ell}\|U^j\|_2^2+\|r_n^\ell\|_2^2+o_n(1),\\
A(v_n)&=\sum_{j=1}^{\ell}A(U^j)+A(r_n^\ell)+o_n(1),\\
C(v_n)&=\sum_{j=1}^{\ell}C(U^j)+C(r_n^\ell)+o_n(1),\\
P(v_n)&=\sum_{j=1}^{\ell}P(U^j)+P(r_n^\ell)+o_n(1),
\end{align*}
where each $o_n(1)$ tends to zero as $n\to\infty$, with $\ell$ fixed.
\end{enumerate}
\end{lemma}

The following energy comparison for merged profiles retains
the dependence of the Kirchhoff term on the total kinetic energy
and will be used to exclude splitting of minimizing sequences.
\begin{theorem}\label{thm:general-merging}
Let $k\in C([0,+\infty);(0,+\infty))$ and define
\[
\widehat k(s):=\int_0^s k(\tau)\,d\tau,\quad s\geq0.
\]
For $v\in H$, set
\[
\mathcal{E}_k(v):=\frac{1}{2}\widehat k(A(v))+\frac{1}{2}C(v)-\frac{1}{p}P(v).
\]
Let $\{U^j\}_{j\geq1}\subset H$ be nonnegative and assume that
\[
\sum_{j\geq1}\left(\|U^j\|_2^2+A(U^j)+C(U^j)+P(U^j)\right)<+\infty.
\]
Set
\[
W:=\left(\sum_{j\geq1}(U^j)^2\right)^{\frac{1}{2}}.
\]
Then, $W\in H$ and
\begin{equation}\label{eq:merging-properties}
\|W\|_2^2=\sum_{j\geq1}\|U^j\|_2^2,\quad
C(W)=\sum_{j\geq1}C(U^j),\quad
A(W)\leq\sum_{j\geq1}A(U^j).
\end{equation}
If at least two of the nonzero functions $U^j$ are positive almost everywhere, then
\begin{equation}\label{eq:merging-P}
P(W)>\sum_{j\geq1}P(U^j),
\end{equation}
and consequently
\begin{equation}\label{strict energy inequality}
\mathcal{E}_k(W)
<\frac{1}{2}\widehat k\left(\sum_{j\geq1}A(U^j)\right)+\frac{1}{2}\sum_{j\geq1}C(U^j)
-\frac{1}{p}\sum_{j\geq1}P(U^j).
\end{equation}
\end{theorem}
\begin{proof}For $N\geq1$, set
\[
W_N:=\left(\sum_{j=1}^{N}(U^j)^2\right)^{1/2}.
\]
For $\varepsilon>0$, define
\[
W_{N,\varepsilon}
:=
\left(\varepsilon^2+\sum_{j=1}^{N}(U^j)^2\right)^{1/2}
-\varepsilon.
\]
The defining map is smooth, vanishes at the origin,
and has bounded first derivatives.
Hence $W_{N,\varepsilon}\in H^1(\mathbb R^3)$, and
the Cauchy--Schwarz inequality gives
\[
A(W_{N,\varepsilon})
=
\int_{\mathbb R^3}
\frac{\left|\sum_{j=1}^{N}U^j\nabla U^j\right|^2}
     {\varepsilon^2+\sum_{j=1}^{N}(U^j)^2}\,dx
\leq \sum_{j=1}^{N}A(U^j).
\]
Since $W_N\in L^2(\mathbb R^3)$,
$0\leq W_{N,\varepsilon}\leq W_N$, and
$W_{N,\varepsilon}\to W_N$ almost everywhere as
$\varepsilon\to0$, by the dominated convergence theorem it follows that
$W_{N,\varepsilon}\to W_N$ in $L^2(\mathbb R^3)$.
By weak compactness and lower semicontinuity,
$W_N\in H^1(\mathbb R^3)$ and
\[
A(W_N)\leq\sum_{j=1}^{N}A(U^j).
\]
Together with $W_N^2=\sum_{j=1}^{N}(U^j)^2$,
this implies that $W_N\in H$ and
\[
\|W_N\|_H^2
\leq\sum_{j=1}^{N}\|U^j\|_H^2
\leq\sum_{j\geq1}\|U^j\|_H^2<\infty.
\]
Thus $\{W_N\}$ is bounded in $H$.

As $W_N\uparrow W$, using the monotone convergence theorem, we get 
\[
\|W\|_2^2=\sum_{j\geq1}\|U^j\|_2^2,
\qquad
C(W)=\sum_{j\geq1}C(U^j).
\]
Moreover, since $0\leq W_N\leq W$,
\[
\|W-W_N\|_2^2
\leq\int_{\mathbb R^3}(W^2-W_N^2)\,dx
=\sum_{j>N}\|U^j\|_2^2
\longrightarrow0.
\]
The boundedness of $\{W_N\}$ in $H$ and this $L^2$ convergence
imply that $W\in H$ and, after passing to a subsequence,
$W_N\rightharpoonup W$ in $H$.
By weak lower semicontinuity and the gradient estimate above,
\[
A(W)\leq\sum_{j\geq1}A(U^j).
\]
This proves \eqref{eq:merging-properties}.

Suppose now that at least two of the functions $U^j$
are positive almost everywhere.
Since $p>2$,
\[
\left(\sum_{j\geq1}(U^j)^2\right)^{p/2}
>
\sum_{j\geq1}(U^j)^p
\quad\text{a.e. in }\mathbb R^3.
\]
Integrating this inequality yields \eqref{eq:merging-P}.
Finally, since $k>0$, its primitive $\widehat k$ is strictly
increasing. Combining this monotonicity with
\eqref{eq:merging-properties} and \eqref{eq:merging-P}
gives \eqref{strict energy inequality}. 
\end{proof}
\begin{remark}
The general formulation above is consistent with the usual Kirchhoff structure. Indeed, for every $v,\varphi\in H$,
\[
\left.\frac{d}{d\varepsilon}\right|_{\varepsilon=0}
\frac12\widehat k\bigl(A(v+\varepsilon\varphi)\bigr)
=k(A(v))\int_{\mathbb R^3}\nabla v\cdot\nabla\varphi\,dx.
\]
Hence, $k(A(v))$ appears as the coefficient of $-\Delta v$ in the
Euler--Lagrange equation associated with $\mathcal E_k$. In particular, the Kirchhoff law in \eqref{1} corresponds to $k(s)=a+bs$, for which $\mathcal E_k=J$.
\end{remark}

The following corollary applies this comparison to minimizing sequences.

\begin{corollary}\label{cor:general-splitting}
Let $k$ be as in Theorem~\ref{thm:general-merging}.
Let $\mathcal A\subset S_\mu$ be nonempty, and set
\[
c:=\inf_{v\in\mathcal A}\mathcal E_k(v)\in\mathbb R.
\]
Suppose that a minimizing sequence
$\{v_n\}\subset\mathcal A$ admits nonnegative profiles
$\{U^j\}_{j\ge1}\subset H$ such that
\[
\sum_{j\ge1}\|U^j\|_2^2=\mu,
\quad
c=\frac12\widehat{k}(K)
+\frac12C_\infty-\frac1pP_\infty,
\]
where
\[
K:=\sum_{j\ge1}A(U^j)<+\infty,\quad
C_\infty:=\sum_{j\ge1}C(U^j)<+\infty,\quad
P_\infty:=\sum_{j\ge1}P(U^j)<+\infty.
\]
Let $W$ be the merged function associated with
$\{U^j\}_{j\ge1}$ as in Theorem~\ref{thm:general-merging}.
If every nonzero profile is positive almost everywhere
and $W\in\mathcal A$, then exactly one profile is nonzero.
\end{corollary}

\begin{proof}
Since $\sum\limits_{j\ge1}\|U^j\|_2^2=\mu>0$, at least one
profile is nonzero.
If two or more profiles were nonzero, their positivity
and Theorem~\ref{thm:general-merging}, together with
$W\in\mathcal A$, would give
\[
c\le\mathcal E_k(W)
<
\frac12\widehat{k}(K)
+\frac12C_\infty-\frac1pP_\infty
=c,
\]
a contradiction.
Hence, exactly one profile is nonzero.
\end{proof}

\begin{remark}
The energy comparison in Theorem~\ref{thm:general-merging} holds for every continuous
positive Kirchhoff function $k$. Applying Corollary~\ref{cor:general-splitting} also
requires the stated mass and energy identities and
$W\in\mathcal{A}$. For $k(s)=a+bs$, these conditions are verified
in the proof of Theorem~\ref{thm:profile-merging}. For a general $k$, they must be
established separately.
\end{remark}

For the mountain-pass construction, we use the following
compactness result from
\cite[Proposition~2.1]{Jeanjean-2025},
applied with $N=3$ and $d=2$.
\begin{lemma}\label{partial-symmetry compactness}Let $\{u_n\}$ converge weakly to $u$ in $H^1(\mathbb R^3)$.
Assume that each $u_n$ is radially symmetric in $x_\perp$,
even in $x_3$, and nonincreasing in $|x_3|$.
Then
\[
u_n\to u\quad\text{strongly in }L^q(\mathbb R^3)
\quad\text{for every }2<q<6.
\]
\end{lemma}

For the mass-preserving scaling $t\star u$ introduced in Section 1,
a change of variables gives
\begin{equation}\label{eq:scaling-identities}
A(t\star u)=t^2A(u),\qquad
C(t\star u)=t^{-2}C(u),\qquad
P(t\star u)=t^{\frac{3p-6}{2}}P(u).
\end{equation}

Let $u\in S_\mu$ be a constrained critical point of $J$,
and let $\lambda$ be its Lagrange multiplier.
Set $M:=a+bA(u)$.
Testing the Euler--Lagrange equation with $u$ gives
\[
MA(u)+C(u)+\lambda\mu=P(u).
\]
Choose $\chi\in C_c^\infty(B_2)$ such that $\chi=1$ on $B_1$.
By local elliptic regularity, we may also test the equation
with $\chi(x/T)x\cdot\nabla u$.
The integrability of $|\nabla u|^2$,
$(1+|x_\perp|^2)|u|^2$ and $|u|^p$ ensures that the terms
involving derivatives of the cutoff vanish as $T\to\infty$.
Using $x\cdot\nabla(|x_\perp|^2)=2|x_\perp|^2$, we obtain
\[
\frac M2 A(u)+\frac52 C(u)+\frac{3\lambda}{2}\mu
=\frac3p P(u).
\]
Eliminating $\lambda$ from these two identities and recalling
that $M=a+bA(u)$, we obtain the Pohozaev identity
\begin{equation}\label{eq:Pohozaev-whole}
aA(u)+bA(u)^2
=
C(u)+\frac{3p-6}{2p}P(u).
\end{equation}

\section{Existence of a local minimizer}\label{Sec:Local}

In this section, we construct a positive local minimizer
of $J$ on $S_\mu$. We first establish the local minimizing
geometry and then prove compactness, up to translations
in the $x_3$-direction, for a suitable minimizing
Palais--Smale sequence.

Fix $\chi>0$ and define
\[
\mathfrak q_a(u):=aA(u)+C(u).
\]
For $r>0$, set
\[
\mathcal V_r:=\{u\in H:\mathfrak q_a(u)<r\},\quad
\overline{\mathcal V}_r:=\{u\in H:\mathfrak q_a(u)\leq r\}.
\]

\begin{lemma}\label{lem:local-geometry}
There exists $\mu_0=\mu_0(a,b,p,\chi)>0$ such that, for every $\mu\in(0,\mu_0)$,
\[
\inf_{u\in S_\mu\cap\mathcal V_{\frac{\chi}{4}}}J(u)
<\inf_{u\in S_\mu\cap(\mathcal V_\chi\setminus\mathcal V_{\frac{\chi}{2}})}J(u).
\]
\end{lemma}
\begin{proof}
Fix $0\leq\zeta\in C_c^\infty(\mathbb R^3)$ with $\|\zeta\|_2=1$, and set
\[
\mu_0:=\min\left\{\left(\frac{p\,a^{\frac{3p-6}{4}}}
{4C_p\,\chi^{\frac{3p-10}{4}}}\right)^{\frac{4}{6-p}},\frac{\chi}{8\mathfrak q_a(\zeta)},
\left(\frac{\chi}{4bA(\zeta)^2}\right)^{\frac12}\right\}.
\]
Let $\mu\in(0,\mu_0)$ and $u\in S_\mu\cap\mathcal V_\chi$. Since $A(u)\leq\frac{\mathfrak q_a(u)}{a}<\frac{\chi}{a}$, using \eqref{eq:GN}, we have
\[
\frac{1}{p}P(u)\leq\frac{C_p}{p}\mu^{\frac{6-p}{4}}
\chi^{\frac{3p-10}{4}}a^{\frac{6-3p}{4}}\,\mathfrak q_a(u)\leq\frac14\mathfrak q_a(u).
\]
Consequently,
\begin{equation}\label{eq:local-lower}
J(u)\geq\frac14\mathfrak q_a(u)+\frac{b}{4}A(u)^2,
\end{equation}
and hence, for every $u\in S_\mu\cap(\mathcal V_\chi\setminus\mathcal V_{\frac{\chi}{2}})$, there holds $J(u)\geq\frac{\chi}{8}$.

Now let $w_\mu:=\sqrt{\mu}\,\zeta$. Then $w_\mu\in S_\mu$ and
\[
\mathfrak q_a(w_\mu)=\mu\mathfrak q_a(\zeta)<\frac{\chi}{8}<\frac{\chi}{4}.
\]
Moreover, we have
\[
J(w_\mu)\leq\frac{\mu}{2}\mathfrak q_a(\zeta)+\frac{b\mu^2}{4}A(\zeta)^2<\frac{\chi}{16}+\frac{\chi}{16}=\frac{\chi}{8}.
\]
Therefore, it follows that
\[
\inf_{u\in S_\mu\cap\mathcal V_{\frac{\chi}{4}}}J(u)\leq J(w_\mu)
<\frac{\chi}{8}\leq\inf_{u\in S_\mu\cap(\mathcal V_\chi\setminus\mathcal V_{\frac{\chi}{2}})}J(u).
\]
\end{proof}

For $\chi>0$ and $\mu\in(0,\mu_0)$, we consider the following minimization problem
\begin{equation*}
c_\mu:=\inf_{u\in S_\mu\cap\mathcal{V}_\chi}J(u).
\end{equation*}
Let
\[
\mu_1:=\min\left\{\mu_0,\,\frac{\sqrt a}{2b}\left(\sqrt{4a^2+\frac{b\chi}{2}}-2a\right)\right\}.
\]
\begin{lemma}\label{lem:below-threshold}
For every $\mu\in(0,\mu_1)$, we have $0<c_\mu<d_\mu$.
\end{lemma}

\begin{proof}
Since $\mu_1\leq\mu_0$, using \eqref{eq:spectral-ineq} with $M=a$ and \eqref{eq:local-lower}, we have
\[
J(u)\geq\frac14\mathfrak q_a(u)\geq\frac{\sqrt a}{2}\mu\quad\text{for }u\in S_\mu\cap\mathcal V_\chi,
\]
thus $c_\mu>0$.

Because $as_\mu^2\leq\mu^2$, the definition of $\mu_1$ gives
\[
2as_\mu+bs_\mu^2\leq2\sqrt a\,\mu+\frac{b}{a}\mu^2<\frac{\chi}{8}.
\]
Therefore, by \eqref{eq:test-A} and \eqref{eq:test-C}, one can show that
\[
\mathfrak q_a(\psi_{\mu,t})=2as_\mu+bs_\mu^2
+a\mu t^2\int_{\mathbb R}|\eta'|^2\,dx_3<\frac{\chi}{4}
\]
for all sufficiently small $t>0$.

Finally, using \eqref{eq:test-A}--\eqref{eq:d-mu-formula}, we get
\begin{align*}
J(\psi_{\mu,t})-d_\mu
&=\frac{\mu}{2}(a+b s_\mu)t^2\int_{\mathbb{R}}|\eta'|^2\,dx_3
+\frac{b\mu^2}{4}t^4\left(\int_{\mathbb{R}}|\eta'|^2\,dx_3\right)^2\notag\\
&\quad-\frac{\mu}{p}s_\mu^{\frac{p-2}{2}}t^{\frac{p-2}{2}}
\left(\int_{\mathbb{R}^{2}}|\Phi_1|^p\,dx_\perp\right)
\left(\int_{\mathbb{R}}|\eta|^p\,dx_3\right).
\end{align*}
Since $\frac{p-2}{2}<2$, the last term dominates as $t\to0^+$. Hence, $J(\psi_{\mu,t})<d_\mu$ for all sufficiently small $t$, and thus $c_\mu<d_\mu$.
\end{proof}

We now have the following nonvanishing property for minimizing sequences.
\begin{lemma}\label{lem:local-nonvanishing}
Let $\mu\in(0,\mu_1)$ and $\{u_n\}\subset S_\mu\cap\mathcal{V}_\chi$ satisfy $J(u_n)\to c_\mu$. Then
\begin{equation*}
\liminf_{n\to\infty}P(u_n)>0.
\end{equation*}
\end{lemma}

\begin{proof}
Suppose, after passing to a subsequence, that $P(u_n)\to0$. By Lemma \ref{lem:vanishing-level},
we have
\begin{equation*}
J(u_n)=J_0(u_n)-\frac1pP(u_n)\geq d_\mu-o(1),
\end{equation*}
contradicting $c_\mu<d_\mu$.
\end{proof}

The preceding estimates exclude vanishing at the level $c_\mu$.
We now rule out profile splitting and prove strong convergence,
up to translations in the $x_3$-direction, for suitable
minimizing Palais--Smale sequences.

\begin{theorem}\label{thm:profile-merging}
Let $\mu\in(0,\mu_1)$. Suppose that $\{u_n\}\subset S_\mu$ is a Palais--Smale sequence for $J|_{S_\mu}$ at level $c_\mu$ and satisfies
\begin{equation}\label{eq:profile-principle-well}
\limsup_{n\to\infty}\mathfrak q_a(u_n)\leq\frac{\chi}{2},\quad\|u_n^-\|_H\to0.
\end{equation}
Then, after passing to a subsequence,  there exist $z_n\in\mathbb{R}$ and a positive function $U\in S_\mu$ such that
\begin{equation*}
\tau_{-z_n}u_n\to U\quad\text{in }H.
\end{equation*}
Moreover, we have
\begin{equation*}
J(U)=c_\mu,\quad\mathfrak{q}_a(U)\leq\frac{\chi}{2}.
\end{equation*}
\end{theorem}

\begin{proof}
Since $\{u_n\}\subset S_\mu$ is a Palais--Smale sequence for $J$, we have
\begin{equation}\label{eq:PS-residual}
-\left(a+bA(u_n)\right)\Delta u_n+|x_\perp|^2u_n+\lambda_nu_n-|u_n|^{p-2}u_n
\to0\quad\text{in }H^*,
\end{equation}
where
\begin{equation*}
\lambda_n:=-\frac{1}{\mu}\left\langle J'(u_n),u_n\right\rangle
=\frac{1}{\mu}\left(P(u_n)-aA(u_n)-bA(u_n)^2-C(u_n)\right).
\end{equation*}
The boundedness of $\{u_n\}$ in $H$ then implies that $\{\lambda_n\}$ is bounded. Passing to a subsequence, we may assume that
\begin{equation*}
A(u_n)\to K_\infty\ge0,\quad
C(u_n)\to C_\infty,\quad
P(u_n)\to P_\infty,\quad
\lambda_n\to\lambda_\infty\in\R,
\end{equation*}
and set
\[
M_\infty:=a+bK_\infty>0.
\]
Using \eqref{eq:profile-principle-well}, we obtain $u_n\in S_\mu\cap\mathcal V_\chi$ for all sufficiently large $n$. From Lemma~\ref{lem:local-nonvanishing}, we have
\[
\liminf_{n\to\infty}P(u_n)>0.
\]
Then, we apply the profile decomposition in Lemma~\ref{lem:profile}. If all profiles were zero, \eqref{eq:remainder-Lq} with $q=p$ would imply $P(u_n)\to0$, a contradiction. Hence, at least one profile is nonzero.

For every nonzero profile $U^j$, testing \eqref{eq:PS-residual} with $\tau_{z_n^j}\varphi$, using \eqref{eq:shifted-weak} and the local compactness of the Sobolev embedding, we obtain
\begin{equation}\label{eq:profile-Euler}
-M_\infty\Delta U^j+|x_\perp|^2U^j+\lambda_\infty U^j=|U^j|^{p-2}U^j
\quad\text{in }\mathbb R^3.
\end{equation}
Since $\|u_n^-\|_H\to0$, $U^j\geq0$. By the strong maximum principle, every nonzero profile is positive. Then, applying Lemma~\ref{lem:strict-edge} to \eqref{eq:profile-Euler}, we can obtain $\lambda_\infty>-2\sqrt{M_\infty}$.

We next show that the profile remainder vanishes. For fixed $\ell\geq1$, testing \eqref{eq:PS-residual} with $u_n$ and  \eqref{eq:profile-Euler} with $U^j$ ($1\leq j\leq\ell$), and using the splittings in Lemma~\ref{lem:profile}, we get
\begin{equation}\label{eq:remainer}
M_\infty A(r_n^\ell)+C(r_n^\ell)+\lambda_\infty\|r_n^\ell\|_2^2-P(r_n^\ell)=o_n(1).
\end{equation}
Then, by \eqref{eq:spectral-ineq}, we can see that
\[
(2\sqrt{M_\infty}+\lambda_\infty)\|r_n^\ell\|_2^2\leq M_\infty A(r_n^\ell)+C(r_n^\ell)
+\lambda_\infty\|r_n^\ell\|_2^2=P(r_n^\ell)+o_n(1).
\]
From \eqref{eq:remainder-Lq}, it follows
\[
\lim_{\ell\to\infty}\limsup_{n\to\infty}\|r_n^\ell\|_2^2=0.
\]
Using \eqref{eq:remainer}, we obtain
\[
\lim_{\ell\to\infty}\limsup_{n\to\infty}\bigl(A(r_n^\ell)+C(r_n^\ell)\bigr)=0.
\]
Consequently, 
\begin{equation}\label{eq:local-full-recovery}
\mu=\sum_{j\geq1}\|U^j\|_2^2,\quad
K_\infty=\sum_{j\geq1}A(U^j),\quad
C_\infty=\sum_{j\geq1}C(U^j),\quad
P_\infty=\sum_{j\geq1}P(U^j).
\end{equation}

We now exclude splitting. Let $\mathcal A:=S_\mu\cap\mathcal V_\chi$ and $k(s):=a+bs$.
Then, one can take $\mathcal E_k=J$ and $\widehat k(s)=as+\frac{b}{2}s^2$. Since $u_n\in\mathcal A$ for all sufficiently large $n$, using \eqref{eq:local-full-recovery}, we have
\[
c_\mu=\frac{1}{2}\widehat k(K_\infty)+\frac{1}{2}C_\infty-\frac{1}{p}P_\infty.
\]
Let $W$ be the merged function associated with
the profiles $\{U^j\}_{j\ge1}$.
By Theorem~\ref{thm:general-merging} and \eqref{eq:local-full-recovery}, we obtain
\[
\|W\|_2^2=\mu,\quad
\mathfrak q_a(W)\leq aK_\infty+C_\infty=\lim_{n\to\infty}\mathfrak q_a(u_n)\leq\frac{\chi}{2}.
\]
Hence, $W\in\mathcal A$. Applying Corollary~\ref{cor:general-splitting}, we conclude that exactly one profile is nonzero.

Relabeling the profiles, we may assume that the unique nonzero profile is $U^1$, and we write $U=U^1$ and $z_n=z_n^1$. By \eqref{eq:local-full-recovery}, we have
$\|U\|_2^2=\mu$, $A(U)=K_\infty$ and $C(U)=C_\infty$.
Taking $\ell=1$ in the splitting identities of Lemma~\ref{lem:profile}, one can see that
\[
\|r_n^1\|_2^2+A(r_n^1)+C(r_n^1)\to0.
\]
Since $u_n=\tau_{z_n}U+r_n^1$ and translations in the $x_3$ direction preserve the $H$-norm, we conclude that
\[
\tau_{-z_n}u_n\to U\quad\text{in }H.
\]
Consequently, we obtain
\[
J(U)=c_\mu,\quad
\mathfrak q_a(U)=\lim_{n\to\infty}\mathfrak q_a(u_n)\leq\frac{\chi}{2}.
\]
\end{proof}

We can now obtain a positive local minimizer by combining
Ekeland's variational principle with the preceding
compactness result.
\begin{proposition}\label{prop:local-minimizer}
For every $\mu\in(0,\mu_1)$, there exists a positive constrained local minimizer $u_1\in S_\mu$ such that
\[J(u_1)=c_\mu<d_\mu.\]
\end{proposition}

\begin{proof}By Lemma~\ref{lem:local-geometry}, we have \[
 c_\mu
 =\inf_{u\in S_\mu\cap\overline{\mathcal V_{\frac{\chi}{2}}}}\, J(u)
 =\inf_{u\in S_\mu\cap\overline{\mathcal V_{\frac{3\chi}{4}}}}\, J(u).
\]
Because $J(|u|)=J(u)$, we may choose a nonnegative minimizing sequence
$\{v_n\}\subset S_\mu\cap\overline{\mathcal V_{\frac{\chi}{2}}}$ for $c_\mu$. Applying Ekeland's variational principle to $J$ on $S_\mu\cap\overline{\mathcal V_{\frac{3\chi}{4}}}$, we obtain another minimizing sequence $\{u_n\}$ for $c_\mu$ satisfying $\|u_n-v_n\|_H\to0$.
Consequently, $\{u_n\}$ is a Palais--Smale sequence for $J|_{S_\mu}$ at level $c_\mu$.

Moreover, since $v_n\ge 0$ and $\|u_n-v_n\|_H\to0$, it follows that $\|u_n^-\|_2\le \|u_n-v_n\|_2\to0$.
Hence, the sequence $\{u_n^-\}$ is bounded in $H$. By the Gagliardo--Nirenberg inequality \eqref{eq:GN}, $\|u_n^-\|_p\to0$.
Since $\{u_n\}$ is a bounded constrained Palais--Smale sequence, there exist bounded multipliers $\{\lambda_n\}$ such that
\[
 J'(u_n)+\lambda_n u_n\to0\quad\text{in }H^*.
\]
Testing this relation against \(u_n^-\) yields
\[
 (a+bA(u_n))A(u_n^-)+C(u_n^-)
 +\lambda_n\|u_n^-\|_2^2
 =\|u_n^-\|_p^p+o(1).
\]
Consequently, $aA(u_n^-)+C(u_n^-)\to0$. Since
$\|u_n^-\|_2\to0$, we conclude that $\|u_n^-\|_H\to0$.
Together with $\limsup\limits_{n\to\infty}\mathfrak q_a(u_n)\le\frac{\chi}2$,
this allows us to apply Theorem~\ref{thm:profile-merging}.
Therefore, after passing to a subsequence, there exist
a sequence $\{z_n\}\subset\mathbb R$ and a positive function
$u_1\in S_\mu$ such that the translates of $u_n$ in the
$x_3$-direction satisfy
\[
\tau_{-z_n}u_n\to u_1
\quad\text{in }H.
\]
By the continuity of $J$ and $\mathfrak q_a$ and their invariance
under these translations, we obtain
\[
J(u_1)=c_\mu,\quad \mathfrak q_a(u_1)\le\frac{\chi}{2}<\chi.
\]
Therefore, \(u_1\) lies in the relative interior of \(S_\mu\cap\mathcal V_\chi\) with respect to \(S_\mu\). Since it attains the infimum of \(J\) on this set, \(u_1\) is a constrained local minimizer of \(J\) on \(S_\mu\).
\end{proof}

\section{Existence of a mountain-pass solution}\label{Sec:Mountain}

In this section, we construct a positive mountain-pass
solution on $S_\mu$ by passing to the limit from solutions
on expanding balls. We begin with the mountain-pass
geometry of $J|_{S_\mu}$.

We define
\[
\mathcal{G}
:=
\left\{
u\in H:
aA(u)+bA(u)^2>\frac{3p-6}{2p}P(u)
\right\},
\qquad
\mathcal{G}_\mu:=\mathcal{G}\cap S_\mu.
\]
By continuity and the scaling identities
\eqref{eq:scaling-identities}, the boundary of $\mathcal{G}_\mu$
relative to $S_\mu$ is given by
\[
\partial\mathcal{G}_\mu
=
\left\{
u\in S_\mu:
aA(u)+bA(u)^2=\frac{3p-6}{2p}P(u)
\right\}.
\]
Every constrained critical point of $J$ on $S_\mu$ belongs to
$\mathcal{G}_\mu$. Indeed, for such a point $u$, we have $C(u)>0$,
since $u\not\equiv0$ and $|x_\perp|^2>0$ almost everywhere.
The Pohozaev identity \eqref{eq:Pohozaev-whole} then gives
\[
aA(u)+bA(u)^2-\frac{3p-6}{2p}P(u)=C(u)>0.
\]\\

We next establish the energy barrier on $\partial G_\mu$
needed for the mountain-pass geometry.
\begin{lemma}\label{lem:MP-barrier}
There exists $\mu_2\in(0,\mu_1]$ such that, for every $\mu\in(0,\mu_2)$, $\mathcal{G}_\mu\neq\varnothing$, $\partial\mathcal{G}_\mu\neq\varnothing$, and
\begin{equation*}\nu_\mu:=\inf_{u\in\partial\mathcal{G}_\mu}J(u)>d_\mu.
\end{equation*}
\end{lemma}

\begin{proof}
For $\psi_{\mu,t}$ defined in \eqref{eq:test-family}, we have
\[
A(\psi_{\mu,t})\to s_\mu,\quad P(\psi_{\mu,t})\to0\quad\text{as }t\to0^+.
\]
Since $as_\mu+bs_\mu^2>0$, one can show that $\psi_{\mu,t}\in\mathcal G_\mu$ for all sufficiently small $t>0$. Thus $\mathcal G_\mu\neq\varnothing$.
Let $w\in\mathcal G_\mu$. Using \eqref{eq:scaling-identities}, we obtain
\[
aA(t\star w)+bA(t\star w)^2-\frac{3p-6}{2p}P(t\star w)=a t^2A(w)+b t^4A(w)^2-\frac{3p-6}{2p}t^{\frac{3p-6}{2}}P(w).
\]
Since $\frac{3p-6}{2}>4$, the right-hand side is negative for all sufficiently large $t$. By continuity, we can show that $t\mapsto t\star w$ intersects $\partial\mathcal G_\mu$. Therefore, $\partial\mathcal G_\mu\neq\varnothing$.

Let $u\in\partial\mathcal G_\mu$. Then, we have
\begin{equation}\label{eq:G-boundary-identity}
aA(u)+bA(u)^2=\frac{3p-6}{2p}P(u).
\end{equation}
Combining \eqref{eq:GN} with \eqref{eq:G-boundary-identity}, it follows that
\[
aA(u)\leq\frac{3p-6}{2p}C_p\mu^{\frac{6-p}{4}}A(u)^{\frac{3p-6}{4}},
\]
and hence
\[
A(u)\geq\left(\frac{2ap}{(3p-6)C_p}\right)^{\frac{4}{3p-10}}\mu^{\frac{p-6}{3p-10}}.
\]
Since $p>\frac{14}{3}$, it follows from \eqref{eq:G-boundary-identity} that, for every $u\in\partial\mathcal G_\mu$
\[
\begin{aligned}
J(u)&=\frac{a(3p-10)}{2(3p-6)}A(u)+\frac{b(3p-14)}{4(3p-6)}A(u)^2+\frac{1}{2}C(u)\geq\frac{a(3p-10)}{2(3p-6)}A(u)\\
&\geq\frac{a(3p-10)}{2(3p-6)}\left(\frac{2ap}{(3p-6)C_p}\right)^{\frac{4}{3p-10}}\mu^{\frac{p-6}{3p-10}}\to+\infty
\end{aligned}
\]
as $\mu\to0^+$. Hence, $\nu_\mu\to+\infty$ as $\mu\to0^+$. In view of \eqref{eq:s-d-asymptotics}, we have $d_\mu\to0$. Consequently, there exists $\mu_2\in(0,\mu_1]$ such that $\nu_\mu>d_\mu$ for every $\mu\in(0,\mu_2)$.
\end{proof}

The following lemma establishes the mountain-pass geometry of $J$ on $S_\mu$.
\begin{lemma}\label{lem:MP-geometry}
Let $\mu\in(0,\mu_2)$. There exist nonnegative functions $w_0\in\mathcal{G}_\mu$, $w_1\in S_\mu\setminus\mathcal{G}_\mu$, both compactly supported in a ball $B_{R_0}$ for some $R_0>0$, such that
\begin{equation*}
m_\mu:=\inf_{\gamma\in\Gamma}\max_{t\in[0,1]}J(\gamma(t))\geq\nu_\mu>d_\mu>\max\{J(w_0),J(w_1)\},
\end{equation*}
where
\[
\Gamma:=\left\{\gamma\in C([0,1],S_\mu):\gamma(0)=w_0,\ \gamma(1)=w_1\right\}.
\]
\end{lemma}
\begin{proof}
As shown in the proof of Lemma~\ref{lem:below-threshold}, $J(\psi_{\mu,t})<d_\mu$ for all sufficiently small $t>0$. Moreover, using \eqref{eq:test-A} and \eqref{eq:test-P}, we get
\[
aA(\psi_{\mu,t})+bA(\psi_{\mu,t})^2-\frac{3p-6}{2p}P(\psi_{\mu,t})\to as_\mu+bs_\mu^2>0
\]
as $t\to0^+$. We may therefore choose $t_\mu>0$ such that $\psi_{\mu,t_\mu}\in\mathcal G_\mu$ satisfying $J(\psi_{\mu,t_\mu})<d_\mu$.

Choose $\vartheta\in C_c^\infty(\mathbb R^3,[0,1])$ such that $\vartheta=1$ on $B_1$ and
$\operatorname{supp}\vartheta\subset B_2$, and set 
\[\vartheta_R(x):=\vartheta(\frac{x}R), \quad w_R:=\frac{\sqrt\mu\,\vartheta_R\psi_{\mu,t_\mu}}{\|\vartheta_R\psi_{\mu,t_\mu}\|_2}.\]
Since
$\vartheta_R\psi_{\mu,t_\mu}\to\psi_{\mu,t_\mu}$ in $H$ as $R\to\infty$, it follows that $w_R\to\psi_{\mu,t_\mu}$ in $H$.
Hence, for all sufficiently large $R$, we have $w_R\in\mathcal G_\mu$ and $J(w_R)<d_\mu$.
Fix such an $R$, set $w_0:=w_R$, and choose $R_0>0$ such that $\operatorname{supp}w_0\subset B_{R_0}$.

For $t\geq1$, we have $t\star w_0\in S_\mu$ and $\operatorname{supp}(t\star w_0)\subset B_{R_0}$. Since $\frac{3p-6}{2}>4$, using \eqref{eq:scaling-identities}, we obtain
\begin{equation*}
J(t\star w_0)=\frac{a}{2}t^2A(w_0)+\frac{b}{4}t^4A(w_0)^2+\frac{1}{2}t^{-2}C(w_0)
-\frac{1}{p}t^{\frac{3p-6}{2}}P(w_0)\to-\infty,
\end{equation*}
and
\[
aA(t\star w_0)+bA(t\star w_0)^2-\frac{3p-6}{2p}P(t\star w_0)\to-\infty
\]
as $t\to\infty$. Thus, for some $t_0>1$, $w_1:=t_0\star w_0$ satisfies $w_1\in S_\mu\setminus\mathcal G_\mu$ and $J(w_1)<d_\mu$. 

The path
\[
\gamma_0(s):=(1+s(t_0-1))\star w_0,\quad s\in[0,1],
\]
belongs to $\Gamma$ and is supported in $B_{R_0}$, so $\Gamma\neq\varnothing$. Since $w_0\in\mathcal G_\mu$ and $w_1\notin\mathcal G_\mu$, every path $\gamma\in\Gamma$ intersects $\partial\mathcal G_\mu$. Therefore, we can obtain
$\max\limits_{s\in[0,1]}J(\gamma(s))\geq\nu_\mu$ for every $\gamma\in\Gamma$,
and thus $m_\mu\geq\nu_\mu$. By Lemma~\ref{lem:MP-barrier} and the construction of $w_0,w_1$, we have
\[
m_\mu\geq\nu_\mu>d_\mu>\max\{J(w_0),J(w_1)\}.
\]
\end{proof}

We next construct positive mountain-pass solutions on balls
containing the supports of $w_0$ and $w_1$.
For $R>R_0$, consider the normalized Dirichlet problem
\begin{equation}\label{eq:bounded-problem}
\begin{cases}
-\left(a+b\displaystyle\int_{B_R}|\nabla u|^2\,dx\right)\Delta u+|x_\perp|^2u+\lambda u=|u|^{p-2}u
&\text{in }B_R,\\
u=0&\text{on }\partial B_R,\\\displaystyle\int_{B_R}|u|^2\,dx=\mu.
\end{cases}
\end{equation}
For each fixed $R$, bounded constrained Palais--Smale
sequences are precompact. The main task is therefore
to obtain a bounded sequence at the mountain-pass level. 
Following \cite{Wang-2024}, we achieve this via a minimax argument that exploits the monotonicity trick.

Define the mass constraint
\[
S_{\mu,R}:=\left\{u\in H_0^1(B_R):\int_{B_R}|u|^2\,dx=\mu\right\}.
\]
For $u\in S_{\mu,R}$, set
\[
T_uS_{\mu,R}:=
\left\{\varphi\in H_0^1(B_R):\int_{B_R}u\varphi\,dx=0\right\}.
\]
For $\rho\in[\frac12,1]$, define $J_{R,\rho}:H_0^1(B_R)\to\mathbb R$ by
\[
J_{R,\rho}(u):=\frac{a}{2}\int_{B_R}|\nabla u|^2\,dx+\frac{b}{4}
\left(\int_{B_R}|\nabla u|^2\,dx\right)^2+\frac{1}{2}\int_{B_R}|x_\perp|^2|u|^2\,dx-\frac{\rho}{p}\int_{B_R}|u|^p\,dx.
\]

Since $w_0$ and $w_1$ are supported in $B_{R_0}$, they belong to $S_{\mu,R}$ for every $R>R_0$. Define
\[
\Gamma_R:=\left\{\gamma\in C([0,1],S_{\mu,R}):\gamma(0)=w_0,\ \gamma(1)=w_1\right\}.
\]
The compactly supported path constructed in the proof of Lemma~\ref{lem:MP-geometry} belongs to $\Gamma_R$, so $\Gamma_R\neq\varnothing$. We then consider the associated constrained minimax problem
\[
m_{\mu,R,\rho}:=
\inf_{\gamma\in\Gamma_R}\max_{t\in[0,1]}J_{R,\rho}(\gamma(t)).
\]
When $\rho=1$, we write
\[J_R:=J_{R,1},\quad m_{\mu,R}:=m_{\mu,R,1}.\]\\

In what follows, we show that the mountain-pass geometry remains valid for $J_{R,\rho}$
when $\rho$ is sufficiently close to $1$, uniformly for $R>R_0$.
\begin{lemma}\label{lem:uniform-bounded-geometry}
There exists $\varepsilon_0=\varepsilon_0(\mu)\in\left(0,\frac{1}{2}\right)$ such that, for every $R>R_0$ and every $\rho\in[1-\varepsilon_0,1]$,
\begin{equation*}
m_{\mu,R,\rho}\geq\nu_\mu>\max\{J_{R,\rho}(w_0),J_{R,\rho}(w_1)\}.
\end{equation*}
\end{lemma}
\begin{proof}
Fix $\gamma\in\Gamma_R$ and extend each $\gamma(t)$ by zero outside $B_R$. Since $w_0$ and $w_1$ are supported in $B_{R_0}\subset B_R$, the extended path belongs to $\Gamma$. Hence it intersects
$\partial\mathcal G_\mu$, and consequently $\max\limits_{t\in[0,1]}J(\gamma(t))\geq\nu_\mu$.
Since $\rho\leq1$, we obtain
\[
J_{R,\rho}(\gamma(t))=
J(\gamma(t))+\frac{1-\rho}{p}P(\gamma(t))\geq J(\gamma(t))\quad\text{for every }t\in[0,1].
\]
Thus, it follows that $\max\limits_{t\in[0,1]}J_{R,\rho}(\gamma(t))\geq\nu_\mu$, which implies $m_{\mu,R,\rho}\geq\nu_\mu$.

Set $\delta_0:=\nu_\mu-\max\{J(w_0),J(w_1)\}>0$ and choose $\varepsilon_0\in(0,\frac12)$ such that
\[
\frac{\varepsilon_0}{p}\max\{P(w_0),P(w_1)\}<\frac{\delta_0}{2}.
\]
Since $w_0$ and $w_1$ are supported in $B_{R_0}$, this choice is independent of $R$. For every $R>R_0$, $\rho\in[1-\varepsilon_0,1]$, and $i\in\{0,1\}$, we get
\[
J_{R,\rho}(w_i)=J(w_i)+\frac{1-\rho}{p}P(w_i)<\max\{J(w_0),J(w_1)\}+\frac{\delta_0}{2}<\nu_\mu.
\]
\end{proof}
To extract a critical point from this minimax construction, we use the following compactness property, whose proof is similar to \cite[Lemma 4]{Wang-2024}.

\begin{lemma}\label{lem:ball-converge}
Let $R>R_0$ and $\rho\in[1-\varepsilon_0,1]$. Then every bounded Palais--Smale sequence for $J_{R,\rho}|_{S_{\mu,R}}$ admits a subsequence converging strongly in $H_0^1(B_R)$.
\end{lemma}

For $\rho\in[1-\varepsilon_0,1]$, let $(u,\lambda)\in S_{\mu,R}\times\mathbb R$ be a positive
constrained critical point of $J_{R,\rho}$. Define the augmented Hessian by
\[
\begin{aligned}
\mathcal Q_{u,\lambda,R,\rho}(\varphi)
&:=\left(a+b\int_{B_R}|\nabla u|^2\,dx\right)\int_{B_R}|\nabla\varphi|^2\,dx
+2b\left(\int_{B_R}\nabla u\cdot\nabla\varphi\,dx\right)^2\\
&\quad+\int_{B_R}|x_\perp|^2\varphi^2\,dx+\lambda\int_{B_R}\varphi^2\,dx
-\rho(p-1)\int_{B_R}u^{p-2}\varphi^2\,dx.
\end{aligned}
\]
We denote by $m^{R,\rho}(u)$ the Morse index of the
augmented Hessian on the full space $H_0^1(B_R)$:
\[
m^{R,\rho}(u)
:=
\sup\left\{
\dim E:
\begin{array}{l}
E\subset H_0^1(B_R)\text{ is a linear subspace},\\
\mathcal Q_{u,\lambda,R,\rho}(\varphi)<0
\text{ for every }\varphi\in E\setminus\{0\}
\end{array}
\right\}.
\]
We refer to this as the free Morse index.
The constrained Morse index is defined by restricting
the same quadratic form to $T_uS_{\mu,R}$.
Since $T_uS_{\mu,R}$ has codimension one in $H_0^1(B_R)$,
a constrained Morse index bound of one implies
$m^{R,\rho}(u)\le2$.
We write
$m^R(u):=m^{R,1}(u)$ and
$\mathcal Q_{u,\lambda,R}:=\mathcal Q_{u,\lambda,R,1}$.\\

We now obtain positive mountain-pass solutions on $B_R$,
together with the uniform bound on their free Morse index
needed in the subsequent compactness argument.
\begin{proposition}\label{prop:uR}
Let $\mu\in(0,\mu_2)$ and $R>R_0$. Then, \eqref{eq:bounded-problem} admits a solution $(u_R,\lambda_R)\in S_{\mu,R}\times\mathbb R$ such that
\begin{equation*}
J_R(u_R)=m_{\mu,R},\quad u_R>0\ \text{in }B_R,
\quad m^R(u_R)\leq2.
\end{equation*}
Moreover, the function $u_R$ is radially symmetric in $x_\perp$, even in $x_3$, and nonincreasing with respect to $|x_\perp|$ and $|x_3|$.
\end{proposition}
\begin{proof}
By modifying the arguments in \cite{Wang-2024} (see also \cite{CJS-2024, CRZ-2025}) and using Lemmas \ref{lem:uniform-bounded-geometry} and \ref{lem:ball-converge}, we can find a sequence $\{\rho_n\}$ with $\rho_n\to1^{-}$ as $n\to\infty$ and a corresponding sequence of nonnegative mountain-pass critical points $\{u_{R,\rho_n}\}$ of $J_{R,\rho_n}$ on $S_{\mu,R}$ at the level $m_{\mu,R,\rho_n}$, satisfying $m^{R,\rho_n}(u_{R,\rho_n})\leq2$.  Using similar arguments as in \cite[Lemma 2.3]{Jeanjean-1999}, we conclude that $m_{\mu,R,\rho_n}\to m_{\mu,R,1}$ as $\rho_n\to1^-$. Note that $u_{R,\rho_n}$ satisfies the following Pohozaev identity
\begin{align*}
&a\int_{B_R}|\nabla u_{R,\rho_n}|^2\,dx
+b\left(\int_{B_R}|\nabla u_{R,\rho_n}|^2\,dx\right)^2
-\int_{B_R}|x_\perp|^2u_{R,\rho_n}^2\,dx\notag\\
&\quad=\left(a+b\int_{B_R}|\nabla u_{R,\rho_n}|^2\,dx\right)
\frac{R}{2}\int_{\partial B_R}\left|\frac{\partial u_{R,\rho_n}}{\partial\nu}\right|^2\,d\sigma
+\frac{3p-6}{2p}\rho_n\int_{B_R}u_{R,\rho_n}^p\,dx.
\end{align*}
Consequently, we have
\begin{align}
m_{\mu,R,\rho_n}
&\geq\frac{3p-10}{6p-12}a\int_{B_R}|\nabla u_{R,\rho_n}|^2\,dx
+\frac{3p-14}{12p-24}b\left(\int_{B_R}|\nabla u_{R,\rho_n}|^2\,dx\right)^2\nonumber\\
&+\frac{3p-2}{6p-12}\int_{B_R}|x_\perp|^2u_{R,\rho_n}^2\,dx.
\label{eq:un-R-bound}
\end{align}
Estimate (\ref{eq:un-R-bound}), together with $p>\frac{14}3$ and
$m_{\mu,R,\rho_n}\to m_{\mu,R}$, shows that
$\{u_{R,\rho_n}\}$ is bounded in $H_0^1(B_R)$.  Moreover,
\[
J_R(u_{R,\rho_n})
=
m_{\mu,R,\rho_n}
-\frac{1-\rho_n}{p}P(u_{R,\rho_n})
\to m_{\mu,R},
\]
and
\[
J_R'(u_{R,\rho_n})
+\lambda_{R,\rho_n}u_{R,\rho_n}
=
(\rho_n-1)|u_{R,\rho_n}|^{p-2}u_{R,\rho_n}
\to0
\quad\text{in }H^{-1}(B_R).
\]
Hence, $\{u_{R,\rho_n}\}$ is a bounded Palais--Smale
sequence for $J_R|_{S_{\mu,R}}$. From Lemma \ref{lem:ball-converge} and the strong maximum principle, there exists $(u_{R},\lambda_{R})\in S_{\mu,R}\times\R$ with $u_R>0$ satisfying \eqref{eq:bounded-problem}.

We then prove $m^R(u_R)\leq2$. Otherwise, there exists a three-dimensional subspace $E\subset H_0^1(B_R)$ such that
\[
J_R''(u_R)[\varphi,\varphi]+\lambda_R\int_{B_R}\varphi^2\,dx<0\quad
\text{for every }0\neq\varphi\in E.
\]
Since $E$ is finite dimensional, the strong convergence of $u_{R,\rho_n}$, together with $\rho_n\to1$ and $\lambda_{R,\rho_n}\to\lambda_R$, implies that for all sufficiently large $n$,
\[
J_{R,\rho_n}''(u_{R,\rho_n})[\varphi,\varphi]+\lambda_{R,\rho_n}\int_{B_R}\varphi^2\,dx<0
\quad\text{for every }0\neq\varphi\in E,
\]
contradicting $m^{R,\rho_n}(u_{R,\rho_n})\leq2$. Hence $m^R(u_R)\leq2$.

Finally, by standard elliptic regularity, it follows that $u_{R}\in C^{2,\eta}\left(\overline{B_R}\right)$ for some $\eta\in(0,1)$. 
We write \eqref{eq:bounded-problem} as
\[
\Delta u_R+F_R(x,u_R)=0\quad\text{in }B_R,
\qquad
u_R=0\quad\text{on }\partial B_R,
\]
where
\[
F_R(x,s):=
\frac{s^{p-1}-(|x_\perp|^2+\lambda_R)s}{a+bA(u_R)},
\qquad x\in B_R,\ s\geq0.
\]
For each $s\geq0$, $F_R(\cdot,s)$ is radially symmetric in
$x_\perp$, nonincreasing in $|x_\perp|$, and independent of $x_3$. Applying \cite[Corollary~1, p.~227]{Gidas-1979}
in every direction in the $x_\perp$-plane and in the
$x_3$-direction, we conclude that $u_R$ is radially symmetric
in $x_\perp$, even in $x_3$, and strictly decreasing in
each of $|x_\perp|$ and $|x_3|$.
\end{proof}

Before taking the limit of the solutions, we show that, as $R\to+\infty$, the mountain-pass levels on $B_R$ converge to the mountain-pass level in the whole space.
\begin{lemma}\label{lem:mR-to-m}
As $R\to+\infty$, we have $m_{\mu,R}\to m_\mu$.
\end{lemma}

\begin{proof}
If $R_0<R<R'$, zero extension gives $\Gamma_R\subset\Gamma_{R'}\subset\Gamma$. Hence, $m_{\mu,R}\geq m_{\mu,R'}\geq m_\mu$. Thus, $m_{\mu,R}$ is nonincreasing in $R$, and there exists $\ell_\mu\geq m_\mu$ such that $m_{\mu,R}\to\ell_\mu$ as $R\to+\infty$.

It remains to prove $\ell_\mu\leq m_\mu$. Fix $\varepsilon>0$ and choose $\gamma\in\Gamma$ such that
\begin{equation*}
\max_{t\in[0,1]}J(\gamma(t))\leq m_\mu+\frac{\varepsilon}{2}.
\end{equation*}
Let $K:=\gamma([0,1])$, which is compact in $H$. Choose the same cutoff $\vartheta$ as in Lemma~\ref{lem:MP-geometry} and define
\begin{equation*}
\xi_R(x):=\vartheta\left(\frac{2x}{R}\right).
\end{equation*}
Then, we have
\begin{equation*}
\xi_R=1\quad\text{on }B_{\frac{R}{2}},
\quad
\operatorname{supp}\xi_R\subset B_R,
\quad
|\nabla\xi_R|\leq\frac{C}{R}.
\end{equation*}
For each fixed $v\in H$,  it follows that
\begin{equation*}
\nabla(\xi_Rv-v)=(\xi_R-1)\nabla v+v\nabla\xi_R.
\end{equation*}
Consequently, $\|\xi_Rv-v\|_H\to0$ as $R\to+\infty$. Since the operators $v\mapsto\xi_Rv$ are uniformly bounded on $H$ and $K$ is compact, we obtain $\sup\limits_{v\in K}\|\xi_Rv-v\|_H\to0$ as $R\to+\infty$.
For $v\in K$, we define
\begin{equation*}
\Pi_R(v):=\frac{\sqrt{\mu}\,\xi_Rv}{\|\xi_Rv\|_2}.
\end{equation*}
With $K\subset S_\mu$, $\Pi_R$ is well defined for all large $R$ and
$\sup\limits_{v\in K}\|\Pi_R(v)-v\|_H\to0$ as $R\to+\infty$.

Since $w_0$ and $w_1$ are compactly supported, for all sufficiently large $R$, we have $\Pi_R(w_0)=w_0$ and $\Pi_R(w_1)=w_1$. Therefore, $\Pi_R\circ\gamma\in\Gamma_R$. Extending $\Pi_R\circ\gamma(t)$ by zero outside $B_R$, it follows that $J_R(\Pi_R\circ\gamma(t))=J(\Pi_R\circ\gamma(t))$.
Thus, for all sufficiently large $R$, 
\[
m_{\mu,R}\leq\max_{t\in[0,1]}J_R(\Pi_R\circ\gamma(t))=\max_{t\in[0,1]}J(\Pi_R\circ\gamma(t))
\leq\max_{t\in[0,1]}J(\gamma(t))+\frac{\varepsilon}{2}\leq m_\mu+\varepsilon,
\]
which gives $\ell_\mu\leq m_\mu+\varepsilon$. Since $\varepsilon>0$ is arbitrary and $\ell_\mu\geq m_\mu$, we deduce $\ell_\mu=m_\mu$.
\end{proof}

We now establish the compactness needed to pass to the limit
in the solutions on expanding balls. Under the assumptions
below, vanishing can occur only at the energy level $d_\mu$.
\begin{theorem}\label{thm:moving-edge}
Let $R_n\to+\infty$. For each $n$, let $(u_n,\lambda_n)\in S_{\mu,R_n}\times\R$ be a positive solution of \eqref{eq:bounded-problem} with $R=R_n$. Assume that the zero extensions of $u_n$ are radially symmetric in $x_\perp$, even in $x_3$, and nonincreasing in $|x_3|$, and that
\[\sup_n\|u_n\|_H<+\infty,\quad m^{R_n}(u_n)\leq2,\quad J(u_n)\to\beta.\]
If $\beta\neq d_\mu$, then, after passing to a subsequence, there exists a positive $u\in S_\mu$ such that $u_n\to u$ in $H$. Moreover, $u$ is a constrained critical point of $J$ and $J(u)=\beta$.
\end{theorem}

\begin{proof}
Extend $u_n$ by zero outside $B_{R_n}$.
By the boundedness of $\{u_n\}$ in $H$ and
Lemma~\ref{partial-symmetry compactness}, after passing to a
subsequence, there exists $u\in H$ such that
\[
u_n\rightharpoonup u\quad\text{in }H,
\qquad
u_n\to u\quad\text{in }L^p(\mathbb R^3).
\]
Testing the equation in $B_{R_n}$ with $u_n$ and using
$\|u_n\|_2^2=\mu$, we obtain
\[
\lambda_n\mu
=
P(u_n)-aA(u_n)-bA(u_n)^2-C(u_n).
\]
Hence, $\{\lambda_n\}$ is bounded.
Passing to a further subsequence, we may assume that
\[
\lambda_n\to\lambda_\infty\in\mathbb R,
\qquad
A(u_n)\to K_\infty\geq0.
\]
Set $M_\infty:=a+bK_\infty>0$.

We first use the bound $m^{R_n}(u_n)\le2$ to establish 
the spectral estimate (\ref{eq:intro-spectral}). Since $M_\infty=a+bK_\infty$, this amounts to proving
\[
\lambda_\infty\ge-2\sqrt{M_\infty}.
\]
Suppose, by contradiction, that this inequality fails. Choose $\delta>0$ such that
\begin{equation*}
\lambda_\infty+2\sqrt{M_\infty}+4\delta<0.
\end{equation*}
By Lemma~\ref{lem:spectral-bottom}, we can choose $0\neq\psi\in C_c^\infty(\mathbb R^2)$ such that
\begin{equation*}
M_\infty\int_{\mathbb R^2}|\nabla_\perp\psi|^2\,dx_\perp
+\int_{\mathbb R^2}|x_\perp|^2\psi^2\,dx_\perp
<\left(2\sqrt{M_\infty}+\delta\right)\int_{\mathbb R^2}\psi^2\,dx_\perp.
\end{equation*}
Fix $0\neq\kappa\in C_c^\infty(\R)$ and set
\[
\kappa_L(x_3):=\kappa\left(\frac{x_3}{L}\right).
\]
Since
\[
\frac{\int_\R|\kappa_L'|^2\,dx_3}{\int_\R|\kappa_L|^2\,dx_3}
=\frac1{L^2}\frac{\int_\R|\kappa'|^2\,dx_3}{\int_\R|\kappa|^2\,dx_3}
\to0\quad\text{as }L\to+\infty,
\]
we may choose $L>0$ so large that
\begin{equation*}
M_\infty\int_{\mathbb R}|\kappa_L'|^2\,dx_3
<\delta\int_{\mathbb R}|\kappa_L|^2\,dx_3.
\end{equation*}
Choose four translations $t_1,\ldots,t_4$ sufficiently far apart so that the functions
\[
\nu_j(x_\perp,x_3):=\psi(x_\perp)\kappa_L(x_3-t_j),\quad j=1,\ldots,4,
\]
have pairwise disjoint supports. For large $n$, all of them belong to $H_0^1(B_{R_n})$. Then
\[
\frac{M_\infty\|\nabla\nu_j\|^2_2+\int_{\R^3}|x_\perp|^2\nu_j^2\,dx}
{\|\nu_j\|_2^2}<2\sqrt{M_\infty}+2\delta.
\]
Since $R_n\to+\infty$, $\int_{B_{R_n}}|\nabla u_n|^2\,dx\to K_\infty$, and $\lambda_n\to\lambda_\infty$, for all sufficiently large $n$ we have
$\nu_j\in H_0^1(B_{R_n})$ and
\begin{equation}\label{Q}
\begin{aligned}
&\left(a+b\int_{B_{R_n}}|\nabla u_n|^2\,dx\right)\int_{B_{R_n}}|\nabla\nu_j|^2\,dx
+\int_{B_{R_n}}|x_\perp|^2\nu_j^2\,dx+\lambda_n\int_{B_{R_n}}\nu_j^2\,dx\\
&\quad\leq\left(\lambda_\infty+2\sqrt{M_\infty}+4\delta\right)\int_{B_{R_n}}\nu_j^2\,dx<0,
\quad j=1,\ldots,4.
\end{aligned}
\end{equation}

Let $Y=\operatorname{span}\{\nu_1,\ldots,\nu_4\}$ and define
\[
L_n(\nu):=\int_{B_{R_n}}\nabla u_n\cdot\nabla\nu\,dx,\quad \nu\in Y.
\]
Then we set $W_n:=Y\cap\ker L_n$. Since $L_n|_Y$ has rank at most one, we have $\dim W_n\geq3$. For $\nu\in W_n$, the positive rank-one Kirchhoff term in the augmented Hessian vanishes. Since the supports of the $\nu_j$ are disjoint, using \eqref{Q}, we obtain that for every $0\neq\nu\in W_n$,
\[
\begin{aligned}
\mathcal Q_{u_n,\lambda_n,R_n}(\nu)
&=\left(a+b\int_{B_{R_n}}|\nabla u_n|^2\,dx\right)\int_{B_{R_n}}|\nabla\nu|^2\,dx \\
&\quad+\int_{B_{R_n}}|x_\perp|^2\nu^2\,dx+\lambda_n\int_{B_{R_n}}\nu^2\,dx \\
&\quad-(p-1)\int_{B_{R_n}}u_n^{p-2}\nu^2\,dx<0.
\end{aligned}
\]
Thus $m^{R_n}(u_n)\geq3$, contradicting $m^{R_n}(u_n)\leq2$. This proves $\lambda_\infty\geq-2\sqrt{M_\infty}$.

We next show that $u\not\equiv0$.
Suppose, by contradiction, that $u\equiv0$.
Since $u_n\to u$ in $L^p(\mathbb R^3)$,
we have $P(u_n)\to0$. Passing to a further subsequence, we may assume that $C(u_n)\to C_\infty$. Testing \eqref{eq:bounded-problem} with $u_n$ and passing to the limit gives
\begin{equation*}
M_\infty K_\infty+C_\infty+\lambda_\infty\mu=0.
\end{equation*}
Since $\lambda_\infty\geq-2\sqrt{M_\infty}$, we have
\begin{equation*}
M_\infty K_\infty+C_\infty\leq2\sqrt{M_\infty}\,\mu.
\end{equation*}
Using \eqref{eq:partial-uncertainty}, it follows that $K_\infty C_\infty\geq\mu^2$, and hence
\begin{equation*}
M_\infty K_\infty+C_\infty
\geq2\sqrt{M_\infty K_\infty C_\infty}
\geq2\sqrt{M_\infty}\,\mu.
\end{equation*}
Thus equality holds throughout. In particular, we have $C_\infty=M_\infty K_\infty$ and $K_\infty C_\infty=\mu^2$, and therefore
\[
aK_\infty^2+bK_\infty^3=\mu^2.
\]
Lemma \ref{lem:vanishing-level} gives the uniqueness of $s_\mu$, so $K_\infty=s_\mu$. Since $P(u_n)\to0$, 
\[
\beta=\lim_{n\to\infty}J(u_n)
=aK_\infty+\frac{3b}{4}K_\infty^2=d_\mu,
\]
contrary to $\beta\neq d_\mu$. Hence, $u\not\equiv0$.

Passing to the limit in \eqref{eq:bounded-problem}, we obtain
\begin{equation}\label{eq:moving-limit-equation}
-M_\infty\Delta u+|x_\perp|^2u+\lambda_\infty u=u^{p-1}
\quad\text{in }\mathbb R^3.
\end{equation}
Since $u\geq0$ and $u\not\equiv0$, we deduce from the strong maximum principle that $u>0$. Combining this with Lemma~\ref{lem:strict-edge}, we obtain $\lambda_\infty>-2\sqrt{M_\infty}$. 

Set $w_n=u_n-u$ and $\delta_\infty=2\sqrt{M_\infty}+\lambda_\infty>0$. Testing \eqref{eq:bounded-problem} with \(u_n\) and \eqref{eq:moving-limit-equation} with \(u\), subtracting the latter identity from the former, and using the convergence \(a+bA(u_n)\to M_\infty\) together with the strong \(L^p\) convergence, we obtain
\begin{equation}\label{wn}
M_\infty A(w_n)+C(w_n)+\lambda_\infty\|w_n\|_2^2=o(1).
\end{equation}
By \eqref{eq:spectral-ineq} and \eqref{wn}, we have
\[
 0\leq\delta_\infty\|w_n\|_2^2\le o(1).
\]
Since $\delta_\infty>0$, it follows that $\|w_n\|_2\to0$.
Returning to \eqref{wn}, we obtain
\[
M_\infty A(w_n)+C(w_n)\to0.
\]
 Hence $u_n\to u$ in $H$. In particular, $\|u\|_2^2=\mu$, $K_\infty=A(u)$, and \eqref{eq:moving-limit-equation} is the original Kirchhoff equation. The strong convergence also gives $J(u)=\beta$.
\end{proof}

\begin{proof}[Completion of the proof of Theorem~\ref{thm:main}]
Set $\mu_*:=\mu_2$. Fix $\mu\in(0,\mu_*)$. By Proposition~\ref{prop:local-minimizer}, we can obtain a positive constrained local minimizer $u_1\in S_\mu$ with $J(u_1)=c_\mu<d_\mu$.

Let $R_n\to+\infty$ and $u_n:=u_{R_n}$ be the positive critical points from Proposition~\ref{prop:uR}. Extend $u_n$ by zero outside $B_{R_n}$. From Lemma~\ref{lem:mR-to-m}, it follows that
\begin{equation*}
J(u_n)=m_{\mu,R_n}\to m_\mu.
\end{equation*}
Using the same Pohozaev argument as in \eqref{eq:un-R-bound}, applied directly to $u_n$, with $\rho=1$, we have
\[
J(u_n)\geq\frac{3p-10}{6p-12}aA(u_n)+\frac{3p-14}{12p-24}bA(u_n)^2+\frac{1}{2}C(u_n).
\]
Since $p>\frac{14}{3}$ and $J(u_n)\to m_\mu$, $\{u_n\}$ is bounded in $H$. 
From Proposition~\ref{prop:uR}, we obtain the symmetry assumptions and the bound $m^{R_n}(u_n)\leq2$ required in Theorem~\ref{thm:moving-edge}. By Lemma \ref{lem:MP-geometry}, we obtain $m_\mu>d_\mu$. We then apply Theorem~\ref{thm:moving-edge} with $\beta=m_\mu$, which yields, after passing to a subsequence, a positive $u_2\in S_\mu$ such that $u_n\to u_2$ in $H$ and $J(u_2)=m_\mu$. Hence, we conclude that $u_2$ is a mountain-pass critical point of $J|_{S_\mu}$. Finally,
\begin{equation*}
J(u_1)=c_\mu<d_\mu<m_\mu=J(u_2),
\end{equation*}
so the two solutions are distinct. This completes the proof.
\end{proof}

\section*{Statements and Declarations}

\subsection*{Funding}
This work is partially supported by the NSFC (12471102), the NSF of Jilin Province (20250102004JC), and the Research Project of the Education Department of Jilin Province (JJKH20250296KJ).

\subsection*{Competing interests}
The authors declare that there is no conflict of interest.

\end{document}